\documentclass[leqno]{amsart}
\usepackage{amssymb}
\usepackage{mathrsfs}
\usepackage{amsmath, amsfonts, vmargin, enumerate}
\usepackage{dsfont}

\newtheorem{thm}{Theorem}
\newtheorem{prop}[thm]{Proposition}
\newtheorem{cor}[thm]{Corollary}
\newtheorem{lem}[thm]{Lemma}
\newtheorem{defi}[thm]{Definition}
\newtheorem{remark}[thm]{Remark}
\newtheorem{example}[thm]{Example}
\newtheorem{pb}[thm]{Problem}

\newenvironment{problem}{\begin{pb}\rm}{\end{pb}}

\newcommand{\real}{{\mathbb R}}

\newcommand{\com}{{\mathbb C}}
\newcommand{\un}{{\mathds {1}}}

\newcommand{\B}{{\mathcal B}}

\newcommand{\F}{{\mathcal F}}

\newcommand{\M}{{\mathcal M}}

\newcommand{\TT}{{\mathcal T}}

\renewcommand{\a}{\alpha}
\renewcommand{\b}{\beta}
\newcommand{\g}{\gamma}
\newcommand{\Ga}{\Gamma}

\renewcommand{\d}{\delta}
\renewcommand{\t}{\theta}
\newcommand{\e}{\varepsilon}
\newcommand{\f}{\varphi}

\renewcommand{\l}{\lambda}
\renewcommand{\O}{\Omega}
\renewcommand{\o}{\omega}
\newcommand{\s}{\sigma}
\newcommand{\Si}{\Sigma}

\newcommand{\8}{\infty}
\newcommand{\el}{\ell}

\newcommand{\wt}{\widetilde}
\newcommand{\wh}{\widehat}
\newcommand{\n}{\noindent}

\newcommand{\les}{\lesssim}

\newcommand{\be}{\begin{align*}}
\newcommand{\ee}{\end{align*}}
\newcommand{\beq}{\begin{equation}}
\newcommand{\eeq}{\end{equation}}
\newcommand{\beqn}{\begin{equation*}}
\newcommand{\eeqn}{\end{equation*}}
\newcommand{\cqd}{\hfill$\Box$}

\begin{document}

\title[Functional calculus and maximal  inequalities]{$H^\8$ functional calculus and maximal  inequalities for semigroups of contractions on  vector-valued $L_p$-spaces}


\author[Q. Xu]{Quanhua Xu}
\address{School of Mathematics and Statistics, Wuhan University, Wuhan 430072, China and Laboratoire de Math{\'e}matiques, Universit{\'e} de Franche-Comt{\'e},
25030 Besan\c{c}on Cedex, France}
\email{qxu@univ-fcomte.fr}

\date{}
\maketitle

\begin{abstract}
 Let $\{T_t\}_{t>0}$ be a strongly continuous semigroup of positive contractions on $L_p(X,\mu)$ with $1<p<\8$. Let $E$ be a UMD Banach lattice of measurable functions on another measure space $(\O,\nu)$. For $f\in L_p(X; E)$ define
 $$\M(f)(x, \o)=\sup_{t>0}\frac1t\Big|\int_0^tT_s(f(\cdot,\o))(x)ds\Big|,\quad (x,\o)\in X\times\O.$$
Then the following maximal ergodic inequality holds
 $$\big\|\M(f)\big\|_{L_p(X; E)}\les \big\|f\big\|_{L_p(X; E)},\quad f\in L_p(X; E).$$
If the semigroup $\{T_t\}_{t>0}$ is additionally assumed to be analytic, then $\{T_t\}_{t>0}$ extends to an analytic semigroup on $L_p(X; E)$ and $\M(f)$ in the above inequality can be replaced by the following sectorial maximal function
 $$\TT_\t(f)(x, \o)=\sup_{|{\rm arg}(z)|<\t}\big|T_z(f(\cdot,\o))(x)\big|$$
for some $\t>0$.

Under the latter analyticity assumption and if $E$ is a complex interpolation space between a Hilbert space and a UMD Banach space, then $\{T_t\}_{t>0}$ extends to an analytic semigroup on $L_p(X; E)$ and its negative generator has a bounded $H^\8(\Sigma_\s)$  calculus for some $\s<\pi/2$.
 \end{abstract}

\bigskip


\section{Introduction}


Let $(X, \F, \mu)$ be  a measure space and $\{T_t\}_{t>0}$ a strongly continuous semigroup of contractions on $L_p(X)$ for every $1\le p\le\8$. Consider the ergodic averages of $\{T_t\}_{t>0}$:
 $$ A_t(T)=\frac1t\int_0^tT_udu$$
and the associated maximal operator:
 $$M(T)(f)=\sup_{t>0}|A_t(T)(f)|.$$
The classical Dunford-Schwartz maximal ergodic inequality asserts that the maximal operator $M(T)$ is bounded on $L_p(X)$ for $1<p\le\8$, and  from $L_1(X)$ to $L_{1,\8}(X)$.

\smallskip

Very recently, Charpentier and Del\'eaval \cite{CD} proved the following vector-valued version of Dunford-Schwartz's inequality: For any $1<q<p<\8$ and any finite sequence $\{f_k\}_{k\ge1}$ in $L_p(X)$
 \beq\label{DS}
  \big\|\big(\sum_k(M(T)(f_k))^q\big)^{1/q}\big\|_p\les\big\|\big(\sum_k|f_k|^q\big)^{1/q}\big\|_p.
  \eeq
Here and in the sequel the symbol $\les$ means an inequality up to a constant depending only on the indices $p, q$, the spaces $E$, etc. but never on the functions  in consideration.

They then asked whether \eqref{DS} remains valid for $1<p<q$. We answer this question by the affirmative.  The proof is simply based on the transference principle that allows us to reduce \eqref{DS} to the special case  where $\{T_t\}_{t>0}$  is the translation group of $\real$. In the latter case, \eqref{DS} is exactly Fefferman-Stein's celebrated vector-valued maximal inequality \cite{FS}.
In fact, we will show  more. To state our result we need recall some definitions. An operator $T$ on $L_p(X)$ is called {\it regular} (more precisely, {\it contractively regular}) if
 $$\big\|\sup_k|T(f_k)|\big\|_p\le \big\|\sup_k|f_k|\big\|_p$$
for all finite sequences $\{f_k\}_{k\ge1}$ in $L_p(X)$. Clearly, any positive contraction $T$ is regular. It is well known that, conversely, if $T$ is regular, then there exists a positive contraction $S$ on  $L_p(X)$  such that $|T(f)|\le S(|f|)$
for any $f\in L_p(X)$; moreover, in this case, $T$ can be written as a linear combination of four positive operators (see \cite[Chapter~1]{MN}).

On the other hand,  it is well known (and easy to check) that if $T$ is a contraction on $L_1(X)$ or $L_\8(X)$, then $T$ is regular. Thus by interpolation,  $T$ is regular on $L_p(X)$ if $T$ is a contraction on $L_p(X)$ for all $1\le p\le\8$.

A regular operator $T$  extends to the vector-valued case. Namely,  $T$ extends to a contraction on $L_p(X;\,E)$ for any Banach space $E$, where $L_p(X;\,E)$ stands for the $L_p$-space of strongly measurable functions from $X$ to $E$. For notational simplicity, this extension will be denoted still by $T$.

\smallskip

We will use UMD Banach lattices. We refer to \cite{bu} for UMD spaces and \cite{LT} for Banach lattices.
Recall that any $L_p$-space with $1<p<\8$ is a UMD space. Let $E$ be a Banach lattice of measurable functions on a measure space $(\O, \nu)$. The functions in $L_p(X; E)$ are viewed as functions of two variables $(x, \o)\in X\times\O$.

\smallskip

Let $\{T_t\}_{t>0}$ be a strongly continuous semigroup of regular contractions on $L_p(X)$. So $\{T_t\}_{t>0}$ extends to a semigroup of contractions on $L_p(X;\,E)$ too. Define
 $$\M(f)(x, \o)=\M(T)(f)(x, \o)=M(T)(f(\cdot, \o))(x),\quad (x,\o)\in X\times\O.$$

\begin{thm}\label{thDS}
 Let $1<p<\8$ and $\{T_t\}_{t>0}$ be a strongly continuous semigroup of regular contractions on $L_p(X)$. Then for any UMD Banach lattice $E$
 \beq\label{DS1}
 \big\|\M(f)\big\|_{L_p(X;\,E)}\les \big\|f\big\|_{L_p(X;\,E)},\quad f\in L_p(X;\,E).
 \eeq
 \end{thm}

Recall that a strongly continuous semigroup $\{T_t\}_{t>0}$ on a Banach space $E$ is called {\it analytic} if there exists $\theta>0$ such that $\{T_t\}_{t>0}$ extends to a bounded analytic function from the sector $\Sigma_\theta$ to $\B(E)$, where $\Sigma_\theta=\{z\in\com: z\neq0,\; |{\rm arg}(z)|<\theta\}$ and $\B(E)$ denotes the space of all bounded linear operators on $E$.

\smallskip

If the semigroup $\{T_t\}_{t>0}$ in Theorem~\ref{thDS} is further assumed to be analytic, the maximal function on the ergodic averages there can be replaced by the maximal function directly taken on the $T_t$'s. Moreover, we can also estimate the following sectorial maximal function for some $\t>0$
 $$\TT_\t(f)(x, \o)=\sup_{z\in\Sigma_\t}\big|T_z(f(\cdot,\o))(x)\big|,\quad (x,\o)\in X\times\O.$$

\begin{thm}\label{thS}
 Let $1<p<\8$ and $\{T_t\}_{t>0}$ be an analytic semigroup of regular contractions on $L_p(X)$. Let $E$ be a UMD lattice on $(\O, \nu)$. Then there exists $\t>0$ such that $\{T_t\}_{t>0}$ extends to a bounded analytic function  on $\Sigma_\t$ with values in $\B(L_p(X; E))$ and
 \beq\label{max}
 \big\|\TT_\t(f)\big\|_{L_p(X;\, E)}\les \big\|f\big\|_{L_p(X;\, E)},\quad f\in L_p(X;\, E).
 \eeq
 \end{thm}

This theorem extends  \cite{LMX} to the vector-valued setting. Like in  \cite{LMX}, we can also prove its discrete analogue. Note that the ancestor of the maximal inequality  \eqref{max} are Stein's  maximal ergodic inequality  for symmetric diffusion semigroups,  i.e., Markovian semigroups of positive contractions on $L_p(X)$ for every $1\le p\le\8$ with $T_t$ selfadjoint on $L_2(X)$ (see  \cite[Chapter~III]{S2}).  Cowling \cite{Cow} proved the sectorial maximal inequality for these semigroups.

\begin{cor}\label{corS}
 Let $E$ be a UMD lattice and $\{T_t\}_{t>0}$ be a semigroup of contractions on $L_p(X)$ for all $1\le p\le\8$. Assume that $\{T_t\}_{t>0}$ is strongly continuous on $L_2(X)$. Then \eqref{DS1} holds for any $1<p<\8$.

 If in addition $\{T_t\}_{t>0}$ is analytic on $L_p(X)$ for some $1<p<\8$, then \eqref{max} holds for any $1<p<\8$.
 \end{cor}

The first part of the corollary follows immediately from Theorem~\ref{thDS}. On the other hand, the second is a consequence of Theorem~\ref{thS} since it is well known that if $\{T_t\}_{t>0}$ is analytic on $L_p(X)$ for one $1<p<\8$, so is it for all $1<p<\8$. The latter fact is easily proved by complex interpolation (see \cite[Chapter~3]{S2}; see also the proof of Lemma~\ref{a-ext} below). The most important case of the corollary is where every $T_t$ is a selfadjoint  operator on $L_2(X)$. Then $\{T_t\}_{t>0}$ is analytic on $L_2(X)$.

\medskip

The proof of Theorem~\ref{thS} is based on  the following result on $H^\8$ functional calculus. We refer to \cite{KW, L1} for $H^\8$  calculus and to  \cite{bl}  for complex interpolation. 

\begin{thm}\label{CH}
 Let $1<p<\8$ and $\{T_t\}_{t>0}$ be an analytic semigroup of regular contractions on $L_p(X)$. Let $E$ be a UMD lattice on $(\O, \nu)$.  Let $(E_0,\, E_1)$ be an interpolation pair of Banach spaces and $0<\eta<1$. Let $E=(E_0,\, E_1)_\eta$ be the associated complex interpolation space. Assume that $E_0$ is isomorphic to a Hilbert space and $E_1$ is a UMD space.
  \begin{enumerate}[\rm(i)]
 \item The extension of $\{T_t\}_{t>0}$ to $L_p(X;\, E)$ is analytic.
 \item  $A$ has a bounded $H^\8(\Sigma_\s)$ functional calculus for some $\s<\pi/2$, where  $-A$ is the generator of $\{T_t\}_{t>0}$ on $L_p(X;\, E)$.
  \end{enumerate}
 In particular, if $E$ is a UMD Banach lattice, then both assertions  hold.
 \end{thm}

The above theorem is proved in \cite{Hy} for symmetric diffusion semigroups. We also refer to \cite{BDT, CL, Fa, T} for some works related to the present article.

We will prove Theorem~\ref{thDS} in section~\ref{Proof of Theorem1}, Theorems~\ref{thS} and \ref{CH} in section~\ref{functional calculus}, and conclude the paper with some further results and open problems.


\section{Proof of Theorem~\ref{thDS}}\label{Proof of Theorem1}


We will prove Theorem~\ref{thDS} in this section. This proof is a simple application of the transference principle. The argument consists in transferring  ergodic inequalities like  \eqref{DS1} to the special case where $\{T_t\}_{t>0}$ is the translation group of $\real$. This powerful technique was invented by Calder\'on \cite{Ca} and  largely developed by Coifman and Weiss \cite{CW}. Since then it is commonly called {\it transference principle} and has been widely applied to many different situations.

 Note that if  $\{T_t\}_t$ is the translation group of $\real$, $M(T)$ becomes the one-sided Hardy-Littlewood maximal function. In the latter case, \eqref{DS1} is Bourgain's vector-valued maximal inequality \cite{B} which extends  Fefferman-Stein's  work. In fact, in this case, the lattice $E$ does not need to be a UMD space. Following \cite{GMT}, $E$ is said to have the {\it Hardy-Littlewood property} if inequality~\eqref{DS1} holds for $\{T_t\}_t$ equal to the translation group of $\real$. Thus the transference argument presented below will show that Theorem~\ref{thDS} remains valid if $E$ has the Hardy-Littlewood property.

 To use transference we first need to dilate our semigroup to  a group of isometries.  Fendler's dilation theorem  is at our disposal for this purpose. It insures that  there exist another larger measure space $(\wt X, \wt\F, \wt\mu)$, a strongly continuous group $\{S_t\}_{t\in\real}$ of regular isometries on $L_p(\wt X)$, a positive isometric embedding $D$ from $L_p(X)$ into $L_p(\wt X)$ and a regular projection $P$ from $L_p(\wt X)$ onto $L_p(X)$ such that
 \beq\label{dilation}
 T_t=PS_tD,\quad\forall\; t>0.
 \eeq
This theorem is proved in \cite{Fe} for positive $T_t$ and then extended to regular $T_t$ in \cite{Fe2}.

We are now ready to do our transference argument. It suffices to prove \eqref{DS1} for $\M_a(T)(f)$ in place of $\M(T)(f)$ for any $a>0$, where
  $$\M_a(T)(f)(x,\o)=\sup_{0<t<a}|A_t(T)(f)(x,\o)|.$$
Let $A(T)(f)=\{A_t(T)(f)\}_{t>0}$.  $A(T)(f)$ is viewed as a function of three variables $(x, \o, t)$ on $X\times\O\times(0,\, \8)$. Then we can write
 $$\big\|\M_a(T)(f)\big\|_{L_p(X;\,E)}=\big\|A(T)(f)\big\|_{L_p(X; \,E(L_\8(0,\,a)))}.$$
Here, given a Banach space $B$, $E(B)$ denotes the space of all strongly measurable functions from $\O$ to $B$ such that $\|f(\cdot)\|_B\in E$. Its norm is defined by $\big\|\|f(\cdot)\|_B\big\|_E$.

By regularity, $P, D$ and $S_t$ all extend to the vector-valued case. By \eqref{dilation}, we have
 $$A(T)=PA(S)D.$$
So
 $$\big\|A(T)(f)\big\|_{L_p(X; \,E(L_\8(0,\,a)))}=\big\|P\big(A(S)(D(f)\big)\big\|_{L_p(X;\,E(L_\8(0,\,a)))}.$$
However,  the regularity of $P$ and $D$ implies
  \begin{align*}
  \big\|P\big(A(S)(D(f)\big)\big\|_{L_p(X;\,E(L_\8(0,\,a)))}
 &\le\big\|A(S)(D(f))\big\|_{L_p(\wt X;\,E(L_\8(0,\,a)))},\\
 \big\|D(f)\big\|_{L_p(\wt X;\,E)}
 &\le\big\|f\big\|_{L_p(X;\,E)}.
  \end{align*}
So we are reduced to proving \eqref{DS1} for $S_t$ in place of $T_t$. Thus we can assume that $\{T_t\}$ itself extends to a group of regular isometries on $L_p(X)$ in the rest of the proof. We will simply write $A_t$ for $A_t(T)$. Then for $s>0$, by the regularity of $T_{-s}$ we have
 \begin{align*}
 \big\|A(f)\big\|_{L_p(X; \,E(L_\8(0,\,a)))}
 &=\big\|A(T_{-s}T_s(f))\big\|_{L_p(X;\,E(L_\8(0,\,a)))}\\
 &=\big\|T_{-s}(A(T_s(f))\big\|_{L_p(X;\,E(L_\8(0,\,a)))}\\
 &\le\big\|A(T_s(f))\big\|_{L_p(X;\,E(L_\8(0,\,a)))}.
 \end{align*}
Thus for any $b>0$ we then deduce
 $$\big\|A(f)\big\|_{L_p(X;\,E(L_\8(0,\,a)))}^p
 \le\frac1b\int_0^b\big\|A(T_s(f))\big\|_{L_p(X;\,E(L_\8(0,\,a)))}^pds.$$
Given $(x,\o)\in X\times\O$  define a function $g(\cdot, x, \o)$ on $\real$ by $ g(s, x,\o)=\un_{(0,\,a+b)}(s)T_s(f)(x,\o)$. Then
 $$ A_t(T_s(f))(x,\o)=\frac1t\int_0^tT_{s+u}(f)(x,\o)du
 =\frac1t\int_0^tg(s+u, x,\o)du,\quad 0<t<a,\; 0<s<b.$$
Therefore,
 $$\M_a(T)(T_s(f))(x,\o)\le M^+(g(\cdot, x, \o))(s)\;{\mathop =^{\rm def}}\;\M^+(g)(s, x,\o),$$
where $M^+$ denotes the usual one-sided Hardy-Littlewood maximal function on $\real$:
 $$M^+(h)(s)=\sup_{t>0}\frac1t\int_0^t|h(s+u)|du,\quad s\in\real.$$
Consequently, by \cite{B} (see also \cite{rubio})
 \begin{align*}
 \int_0^b\big\|A(T_s(f))(x,\cdot)\big\|_{E(L_\8(0,\,a))}^pds
 &=\int_0^b\big\|\M_a(T)(T_s(f))(x, \cdot)\big\|_{E}^pds\\
 &\le\int_{\real}\big\|\M^+(g)(s, x, \cdot)\big\|_{E}^pds\\
 &\lesssim  \int_{\real}\big\|g(s, x, \cdot)\big\|_{E}^pds\\
 &=\int_{0}^{a+b}\big\|T_s(f)(x, \cdot)\big\|_{E}^pds.
 \end{align*}
Taking integral over $X$ and using the regularity of $T_s$, we then get
 \begin{align*}
 \int_0^b\int_X\big\|\M_a(T)(T_s(f))(x,\cdot)\big\|_{E}^pd\mu(x)ds
 &\le\int_{0}^{a+b}\int_X\big\|T_s(f)(x, \cdot)\big\|_{E}^pd\mu(x)ds\\
 &\le(a+b)\big\|f\big\|_{L_p(X;\,E)}^p.
  \end{align*}
Combining the preceding inequalities, we finally obtain
 $$\big\|\M_a(T)(f)\big\|_{L_p(X;\,E)}^p
 \lesssim \frac{a+b}b\big\|f\big\|_{L_p(X;\,E)}^p.$$
Letting $b\to\8$ yields the desired inequality. The theorem is thus proved.
 \cqd


\section{$H^\8$ functional calculus}\label{functional calculus}


We will prove Theorems~\ref{thS} and \ref{CH} in this section. Throughout the section we will fix an analytic semigroup $T=\{T_t\}_{t>0}$ of regular contractions on $L_p(X)$ with $1<p<\8$. So  $T: \Si_\t\to\B(L_p(X))$ is an analytic function for some positive angle $\t$ and
 \beq\label{a-bound}
 \sup_{z\in\Si_\t}\big\|T_z\big\|_{\B(L_p(X))}\le C<\8.
 \eeq

Let $H^\8(\Si_\s)$ be the Banach space of all bounded analytic functions on $\Si_\s$ equipped with the uniform norm $\|\,\|_\8$. Recall that a sectorial operator $A$ of type $\t$ on a Banach space $E$ is said to have  a {\it bounded $H^\8(\Si_{\s})$ calculus }with $\s>\t$ if there exists a constant $C$ such that for any $\f\in H^\8(\Si_\s)$, $\f(A)$ is a well defined (unique) operator on $E$ and
 $$\big\|\f(A)\big\|_{\B(E)}\le C\|\f\|_\8.$$

The proof of Theorem~\ref{CH} requires two lemmas.

\begin{lem}\label{a-ext}
Let $(E_0,\, E_1)$ be an interpolation pair of Banach spaces with $E_0$  isomorphic to a Hilbert space. Let $E=(E_0,\, E_1)_\eta$ with $0<\eta<1$. Then the extension of $\{T_t\}_{t>0}$ to $L_p(X;\, E)$ is analytic. In particular, if E is a Banach lattice with nontrivial convexity and concavity, then the conclusion holds.
 \end{lem}

\begin{proof}
First observe that the assertion  is obvious  for $E=\el_p$. In this case  \eqref{a-bound} gives
 $$\sup_{z\in\Si_\t}\big\|T_z\big\|_{\B(L_p(X;\,\el_p))}\le C.$$
We will then show the assertion  for $E=\el_q$ with any $1<q<\8$ by interpolation. Assume that $p<q$ for the moment. Then $\el_q=(\el_\8,\,\el_p)_{p/q}$. Interpolating the above inequality with the following
 $$\sup_{t>0}\big\|T_t\big\|_{\B(L_p(X;\,\el_\8))}\le 1,$$
we deduce
 $$\sup_{z\in\Si_{\t p/q}}\big\|T_z\big\|_{\B(L_p(X;\,\el_q))}\le C^{p/q}.$$
Thus $T: \Si_{\t p/q}\to\B(L_p(X;\,\el_q))$ is a bounded analytic function. As this interpolation argument is used several times in the sequel, we give the details for the reader's convenience. The change of variables $z=e^{i\zeta}$ maps the sector $\Si_\t$ to the vertical trip $S_\t=\{\zeta=u+iv: |u|<\t\}$. Now fix a point $\zeta_0=u_0+iv_0\in S_{\t p/q}$ with $u_0\neq0$. Choose $u_1$ such that $u_0=u_1p/q$ and $|u_1|<\t$.  Let $f$ be in the open unit ball of  $L_p(X;\,\el_q)$. Since
 $$L_p(X;\,\el_q)=\big(L_p(X;\,\el_\8),\,L_p(X;\,\el_p)\big)_{p/q},$$
there exists a continuous function $\f$ on the closed strip $\{\zeta=u+iv: 0\le u\le1\}$, analytic in the interior such that $\f(p/q)=f$ and
 $$\sup_{v\in\real}\big\|\f(iv)\big\|_{L_p(X;\,\el_\8)}\le1,\quad \sup_{v\in\real}\big\|\f(1+iv)\big\|_{L_p(X;\,\el_p)}\le1.$$
Now define another analytic function $\psi$ by
 $$\psi(\zeta)=T_{e^{iu_1\zeta-v_0}}(\f(\zeta)),\quad \zeta=u+iv,\; 0\le u\le 1.$$
Then
 $$\sup_{v\in\real}\big\|\psi(iv)\big\|_{L_p(X;\,\el_\8)}\le1,\quad \sup_{v\in\real}\big\|\psi(1+iv)\big\|_{L_p(X;\,\el_p)}\le C.$$
Since $\psi(p/q)=T_{e^{i\zeta_0}}(f)$, we then deduce
 $$T_{e^{i\zeta_0}}(f)\in  L_p(X;\,\el_q)\;\textrm{ and }\; \big\| T_{e^{i\zeta_0}}(f)\big\|_{L_p(X;\,\el_q)}\le C^{p/q}.$$
Taking the supremum over $f$ in the unit ball of $L_p(X;\,\el_q)$ yields
 $$\big\| T_{e^{i\zeta_0}}\big\|_{\B(L_p(X;\,\el_q))}\le C^{p/q},$$
which is the desired inequality.

The same argument applies to the case $q<p$ with $\el_\8$ replaced by $\el_1$.

In particular, our assertion holds for $E=\el_2$, so for any Hilbert space $H$ too. Now if $E=(E_0,\, E_1)_\eta$ with $E_0$ isomorphic to a Hilbert space, then the above interpolation argument yields the analyticity of $\{T_t\}_{t>0}$ on $L_p(X;\, E)$.

The last part of the lemma follows from Pisier's theorem \cite{pis} which asserts that every Banach lattice $E$ with nontrivial convexity and concavity is isomorphic to a complex interpolation space between $L_2$ and another Banach lattice.
 \end{proof}

The following elementary lemma is taken from \cite{KKW}.

\begin{lem}\label{stability HC}
 Let $(E_0,\, E_1)$ be an interpolation pair of Banach spaces and $E=(E_0,\, E_1)_\eta$ with $0<\eta<1$.  Assume that a sectorial operator $A$ has a bounded $H^\8(\Si_{\s_j})$ calculus on $E_j$ for $j=0, 1$. Then $A$ has a bounded $H^\8(\Si_{\s})$ calculus on $E$ for any $\s>\s_\eta=(1-\eta)\s_0+\eta\s_1$.
 \end{lem}

\medskip\n{\bf Proof of Theorem~\ref{CH}.} Part (i) is already contained in Lemma~\ref{a-ext}. It remains to prove (ii). To this end, first note that if $E$ is a UMD space, then $A$ has a bounded $H^\8(\Si_{\s_1})$  calculus for any $\s_1>\pi/2$ thanks to \cite{HP} (see also \cite[Corollary~10.15]{KW}).  Note that this result is stated and proved  for positive semigroups in these papers. But the proof requires only the existence of a dilation to groups. The later is insured by Fendler's dilation theorem \cite{Fe2} for regular semigroups too. On the other hand, it is known that $A$ has a bounded $H^\8(\Si_{\s_0})$  calculus for some $\s_0<\pi/2$ on $L_p(X)$, i.e. in the scalar-valued case with $E=\com$ (see \cite[Proposition~2.8]{LMX}). Thus $A$ has a bounded $H^\8(\Si_{\s_0})$  calculus on $L_p(X;\el_p)$ too. Now  choose appropriate $q\in (1,\, \8)$ and $\eta\in (0,\,1)$ such that  $1/2=(1-\eta)/p + \eta/q$. Then by Lemma~\ref{stability HC}, we deduce that $A$ has a bounded $H^\8(\Si_\s)$  calculus for some $\s<\pi/2$ on $L_p(X;\el_2)$, so on  $L_p(X;H)$ for any Hilbert space $H$ too. Finally, a second application of Lemma~\ref{stability HC} finishes the proof of (ii).

If $E$ is a UMD Banach lattice on $\O$, then by \cite{rubio} there exists another UMD lattice $F$ such that $E=(L_2(\O),\, F)_{\eta}$ with $0<\eta<1$, so $E$ satisfies the assumption of the theorem. \cqd

\medskip\n{\bf Proof of Theorem~\ref{thS}.} Using Theorems~\ref{thDS} and \ref{CH}, we can easily adapt the proof of \cite[Theorem~7]{Cow} to the present setting. Let us give the main lines for the reader's convenience. Given $\f\in (-\frac\pi2,\, \frac\pi2)$ define
 $$\Phi_\f(\l)=\exp\big(-e^{{\rm i}\f}\l\big)-\int_0^1e^{-t\l}\,dt,\quad \l>0.$$
Let $\Psi_\f=\Phi_\f\circ\exp$. The Fourier transform of $\Psi_\f$ satisfies the following estimate (see \cite{Cow}):
 $$|\wh\Psi_\f(u)|\les e^{(|\f|-\frac\pi2)|u|},\quad u\in\real.$$
By the Fourier inversion formula, 
 $$ \Phi_\f(\l)=\frac1{2\pi}\int_{\real}\wh\Psi_\f(u) \l^{{\rm i}u}\,du.$$
Now let $z=se^{{\rm i}\f}\in\Si_\t$. We then have
  \begin{align*}
  T_z
  &=e^{-zA}= \Phi_\f(sA)+ \int_0^1e^{-tsA}\,dt\\
  &=\frac1{2\pi}\int_{\real}\wh\Psi_\f(u) s^{{\rm i}u}A^{{\rm i}u}\,du +\frac1s\int_0^sT_t\,dt.
  \end{align*}
It thus follows that
 $$\mathcal T_\t(f)\les \int_{\real}e^{(\t-\frac\pi2)|u|}\,|A^{{\rm i}u}(f)|\,du+\M(f).$$
The second term on the right hand side is estimated by Theorem~\ref{thDS}. For the first we use Theorem~\ref{CH} to conclude that $A$ has bounded imaginary powers of angle $\s<\pi/2$:
  $$\big\|A^{iu}\big\|_{\B(L_p(X;\, E))}\les e^{\s |u|}.$$
Therefore, if  $\t<\frac\pi2-\s$, we get the desired estimate for the first term too. \cqd


\section{More remarks and problems}


The following individual convergence theorem is an easy consequence of Theorem~\ref{thS}. Recall that the fixed point subspace of the semigroup $\{T_t\}_{t>0}$ is 
 $$\F=\{f\in L_p(X;\, E): T_t(f)=f,\; \forall t>0\}.$$
It is well known that $\F$ is complemented in $L_p(X;\, E)$ and the map
 $$P(f)=\lim_{t\to\8}\frac1t\int_0^tT_s(f)ds$$
defines a contractive projection from $L_p(X;\, E)$ onto $\F$ (see \cite[Chapter~VIII]{DS}).

\begin{prop}\label{individual}
 Keep the assumption and notation of Theorem~\ref{thS}. Then for any $f\in L_p(X;\, E)$
 $$\lim_{\Sigma_\t\ni z\to 0}T_z(f)=f\quad\textrm{and}\quad\lim_{\Sigma_\t\ni z\to \8}T_z(f)=P(f)\quad a.e. \textrm{ on } X\times\O.$$
 \end{prop}

\begin{proof}
 Let $g\in L_p(X)$ and $s>0$. Let $\g$ be a circle of center $s$ and radius $r$ with $r<s\sin\t$. For any $z$ inside $\g$ by the Cauchy formula we have
 $$T_z(g)=\frac1{2\pi i}\int_\g \frac{T_\zeta(g)d\zeta}{\zeta-z}.$$
Thus
 $$T_z(g)-T_s(g)=\frac{z-s}{2\pi i}\int_\g \frac{T_\zeta(g)d\zeta}{(\zeta-z)(\zeta-s)}.$$
Consequently,
 $$
 |T_z(g)-T_s(g)|
 \le \frac{|z-s|}{r\pi}\int_\g \frac{|T_\zeta(g)|}{|\zeta-s|}\,d|\zeta|,\quad |z-s|<\frac{r}2. $$
Note that the last integral is a function in $L_p(X)$. Therefore
 $$\lim_{z\to s}T_z(g)=T_s(g)\quad a.e.\textrm{ on } X.$$
It then follows that $\lim_{z\to 0}T_z(T_s(g))=T_s(g)$ a.e..

Let $F$ be the linear span of $\{T_s(g)\otimes h: g\in L_p(X), h\in E, s>0\}$. Then $F$ is dense in $L_p(X; E)$ and by what is proved above $\lim_{z\to 0}T_z(f)=f$ a.e. on $X\times\O$ for any $f\in F$. The assertion then follows from \eqref{max}. Indeed, taking a sequence $(f_n)$ in $F$ such that $f_n\to f$ in $L_p(X; E)$, we have
 $$\limsup_{\Sigma_\t\ni z\to 0}|T_z(f)-f|\le \TT_\t(f-f_n)+|f-f_n|.$$
Thus  by \eqref{max},
 $$\big\|\limsup_{\Sigma_\t\ni z\to 0}|T_z(f)-f|\big\|_{L_p(X; E)}\les \|f-f_n\|_{L_p(X; E)}.$$
Letting $n\to\8$, we deduce that $\limsup_{\Sigma_\t\ni z\to 0}|T_z(f)-f|=0$ a.e..

We turn to the second limit. Let $-A$ be the generator of  $\{T_t\}_{t>0}$. Then $L_p(X;\, E)$ is decomposed into the direct sum of the null and range spaces of $A$: $L_p(X;\, E)=\mathcal N(A)\oplus\overline{\mathcal R(A)}$. Moreover, $\mathcal N(A)$ is the fixed point subspace of the semigroup $\{T_t\}_{t>0}$. Thus it suffices to prove that for any $f\in \overline{\mathcal R(A)}$
 $$\lim_{\Sigma_\t\ni z\to \8}T_z(f)=0\quad a.e. \textrm{ on } X\times\O.$$
Using  \eqref{max} as in the previous part of the proof, we need only to do this for $f$ in a dense subset of $\overline{\mathcal R(A)}$. It is well known that $\{T_{t+s}(g)-T_s(g): s>0, t>0, g\in L_p(X; E)\}$ is such a subset (see \cite[Chapter~VIII]{DS}). So we are reduced to proving the above limit for $f=T_{t+s}(g)-T_s(g)$. To this end, we will use the integral representation of $T_z$. Let $0<\s<\frac\pi2-\t$ and $\d>0$ be sufficiently small. Let $D(0, \d)$ be the disc of center the origin and radius $\d$. Let $\Ga_\d$ be the closed path consisting of the part of the boundary of $\Si_\s$ outside of $D(0, \d)$ and the part of the boundary of $D(0, \d)$  outside of $\Si_\s$. Then
 $$T_z=\frac1{2\pi {\rm i}}\int_{\Ga_\d}e^{-z\l}R(\l, A)d\l,\quad z\in\Si_\t,$$
where $R(\l, A)=(\l-A)^{-1}$. Thus for $f=T_{t+s}(g)-T_s(g)$ as above, we have
 $$
 T_z(f)=\frac1{2\pi {\rm i}}\int_{\Ga_\d}\big(e^{-(z+t+s)\l}-e^{-(z+s)\l}\big)R(\l, A)(g)d\l.
 $$
Let $\Ga$ be the boundary of $\Si_\s$, By the sectoriality of $A$,  $\|\l R(\l, A)\|$ is bounded on $\Ga$. So the above integral is absolutely convergent on $\Ga$. Thus letting $\d\to0$ we deduce
 $$
 T_z(f)=\frac1{2\pi {\rm i}}\int_{\Ga}\big(e^{-(z+t+s)\l}-e^{-(z+s)\l}\big)R(\l, A)(g)d\l.
 $$
Hence for $1<q<\8$ with conjugate index $q'$ we have
 \begin{align*}
 |T_z(f)|
 &\le\frac{t}{2\pi}\int_{\Ga}|e^{-(z+t+s)\l}|\,|\l R(\l, A)(g)|\,|d\l|\\
 &\le\frac{t}{2\pi}\Big(\int_{\Ga}|e^{-q'z\l}|\,|d\l|\Big)^{1/q'}\,
 \Big(\int_{\Ga}|e^{-q(t+s)\l}|\,|\l R(\l, A)(g)|^q\,|d\l|\Big)^{1/q}\\
 &\les\frac1{|z|^{1/q'}}\,\Big(\int_{\Ga}|e^{-q(t+s)\l}|\,|\l R(\l, A)(g)|^q\,|d\l|\Big)^{1/q}.
 \end{align*}
Since the lattice $L_p(X;\, E)$ has the UMD property, it is $q$-convex for some $1<q<\8$.  Then for such a choice of $q$, the last integral represents a function in $L_p(X;\, E)$. Therefore,
 $$\lim_{\Sigma_\t\ni z\to \8}T_z(f)=0\quad a.e. \textrm{ on } X\times\O,$$
 as desired. \end{proof}

Similarly, we have the following individual ergodic theorem corresponding to Theorem~\ref{thDS}.

\begin{prop}
 Under the assumption of Theorem~\ref{thDS}, we have
 $$\lim_{t\to0}A_t(T)(f)=f\quad\textrm{and}\quad\lim_{t\to\8}A_t(T)(f)=P(f)\quad a.e. \textrm{ on } X\times\O$$
for any $f\in L_p(X; E)$.
 \end{prop}

 \begin{proof}
  Without loss of generality, we can assume that $T_t$ is positive for all $t$. Let $A_t=A_t(T)$. By virtue of the maximal inequality \eqref{DS1}, it suffices to show the first limit for $f$ in the domain of $A$, for instance, for $f=T_s(g)$ with $s>0$ and $g\in L_p(X; E)$. We can further assume that $f\ge0$. Then
  $$A_t(f)-f=-\frac1t\int_0^t(t-u)T_u(A(f))du.$$
Thus
 $$|A_t(f)-f|\le\int_0^tT_u(A(f))du\le t\,\M(A(f)).$$
Since $\M(A(f))\in L_p(X; E)$, we deduce that $\lim_{t\to0}A_t(T)(f)=f$ a.e..

As in the proof of the previous proposition, we need only to show the second limit for $f=T_s(g)-g$ with $s>0$ and $g\in L_p(X; E)$. For such an $f$ we have
 $$A_t(f)=\frac1t\int_t^{t+s}T_u(g)du-\frac1t\int_0^sT_u(g)du.$$
Hence for $t$ sufficiently large
 $$|A_t(f)|\le2 \int_t^{t+s}\frac1u\, |T_u(g)|\,du+ \frac{s}t\,\M(g).$$
The second term on the right hand side tends to $0$ as $t\to\8$. To treat the first one, we use again the $q$-convexity of $L_p(X; E)$. Choose $\a\in (0,\,1)$ such that $\a q>1$. Let $\b=1-\a$. Then
 \begin{align*}
 \int_t^{t+s}\frac1u\, |T_u(g)|\,du
 &\le \Big(\int_t^{t+s}\frac1{u^{\b q'}} \,du\Big)^{1/q'}\Big(\int_t^{t+s}\frac1{u^{\a q}}\, |T_u(g)|^q\,du\Big)^{1/q}\\
 &\le \frac1{(1-\b q')^{1/q'}}\,\Big((t+s)^{1-\b q'}-t^{1-\b q'}\Big)^{1/q'}
 \Big(\int_1^{\8}\frac1{u^{\a q}}\, |T_u(g)|^q\,du\Big)^{1/q}.
 \end{align*}
 The $q$-convexity of $L_p(X; E)$ implies that the last integral represents a function belonging to $L_p(X; E)$. On the other hand, the factor in the front of this integral tends to $0$ as $t\to0$. We then deduce that  $\lim_{t\to0}A_t(T)(f)=0$ a.e. on $X\times\O$.
 \end{proof}

The following result falls in the line of investigation of the so-called vector-valued Littlewood-Paley-Stein theory, it emerged from \cite{X} and since then has been significantly developed. \cite{Hy, MTX} are two subsequent works directly related to the subject of this paper.

 \begin{prop}\label{individual}
 Under the assumption of Theorem~\ref{thS} we have the following square function inequality
 \beq\label{square}
  \left\|\Big(\int_0^\8t\big|\frac{\partial}{\partial t}T_t(f)dt\big|^2\Big)^{1/2}\right\|_{L_p(X;\, E)}
 \les \big\|f\big\|_{L_p(X;\, E)},\quad f\in L_p(X;\, E).
 \eeq
 \end{prop}

\begin{proof}
 It is well known that a bounded $H^\8(\Sigma_\s)$ calculus with $\s<\pi/2$ implies  square function inequalities like \eqref{square}. Let us precise this. Let $\s'$ be another angle such that $\s<\s'<\pi/2$. Let $\f$ be a bounded analytic function on  $\Si_{\s'}$. Then by $H^\8$ calculus $\f(tA)$ is a well-defined bounded operator on $L_p(X;\, E)$ for every $t>0$ and we have the following square function inequality
 \beq\label{square1}
 \left\|\Big(\int_0^\8\big|\f(tA)(f)\big|^2\Big)^{1/2}\frac{dt}{t}\right\|_{L_p(X;\, E)}\les \big\|f\big\|_{L_p(X;\, E)},\quad f\in L_p(X;\, E).
 \eeq
This fundamental result is proved in \cite{CDMY} in the case $E=\com$, i.e., for the space $L_p(X)$ (see Corollary~6.7 there). As observed in \cite{3L} (see Lemma~5.3 there; see also \cite{L2}),  the proof of \cite{CDMY} is valid without change for any Banach lattice with nontrivial concavity in place of $L_p(X)$. The space we are concerned here is $L_p(X;\, E)$ which has both nontrivial concavity and convexity.  Thus \eqref{square1} holds.

Now let $\f$ be the function $\f(z)=ze^{-z}$. Then
 $$\f(tA)(f)=tAe^{-tA}(f)=-t \frac{\partial}{\partial t}T_t(f).$$
So \eqref{square1} becomes \eqref{square} for this special choice of $\f$.
 \end{proof}

We conclude this section with some open problems. The first one concerns the weak type $(1, 1)$ version of inequality~\eqref{DS1} under the assumption of Corollary~\ref{corS}.

\begin{problem}
 Let $E$ be a UMD lattice and $\{T_t\}_{t>0}$ a semigroup of contractions on $L_p(X)$ for every $1\le p\le\8$. Does one have
 $$\big\|\M(f)\big\|_{L_{1,\8}(X;\,E)}\les \big\|f\big\|_{L_1(X;\,E)},\quad f\in L_1(X;\,E)?$$
This problem is open even for $E=\ell_q$ with $1<q<\8$.
 \end{problem}

 On the other hand, it would be  interesting to determine the family of Banach spaces $E$ satisfying (i) (resp. (ii)) of Theorem~\ref{CH}.  It is easy to check that both families are closed under the passage to subspaces and quotient spaces. Pisier \cite{pis2} proved that if $E$ is  of nontrivial type, then  $E$ satisfies (i)  for any symmetric convolution semigroup $\{T_t\}_{t>0}$  on a locally compact abelian group.

\begin{problem}
 Let $1<p<\8$ and $\{T_t\}_{t>0}$ be an analytic semigroup of regular contractions on $L_p(X)$.
 \begin{enumerate}[(i)]
 \item Let $E$ be a Banach space of nontrivial type. Does $\{T_t\}_{t>0}$ extend to an analytic semigroup on $L_p(X;\, E)$?
 \item Let $E$ be a UMD Banach space. Does $\{T_t\}_{t>0}$ extend to an analytic semigroup on $L_p(X;\, E)$?  If Yes,  does its negative generator $A$ have a bounded $H^\8(\Sigma_\s)$ functional calculus on $L_p(X;\, E)$ for some $\s<\pi/2$?
 \end{enumerate}
 \end{problem}

Both parts remain open even in the most important case where $\{T_t\}_{t>0}$ is a  symmetric diffusion semigroup (see \cite[Remark~1.8]{pis2}).  For such a semigroup the extension of $\{T_t\}_{t>0}$ to $L_p(X;\, E)$ is analytic for any $1<p<\8$ and for any UMD space $E$ (in fact, only the superreflexivity of $E$ is required; see \cite[Remark~1.8]{pis2}). However,  even in this special case,  it is still an open problem whether $A$ has a bounded $H^\8(\Sigma_\s)$ functional calculus for some $0<\s<\pi/2$. As already observed at the beginning of the proof of Theorem~\ref{CH}, $A$ has a bounded $H^\8(\Sigma_\s)$ functional calculus for some $\s>\pi/2$. Thus by \cite[Theorem~5.3]{KaW}, the last problem is equivalent to the R-analyticity of $\{T_t\}_{t>0}$ on $L_p(X;\, E)$. Let us record this explicitly here:

\begin{problem}
 Let $\{T_t\}_{t>0}$  be a  symmetric diffusion semigroup.
Does $\{T_t\}_{t>0}$ extend to an R-analytic semigroup on $L_p(X;\, E)$ for every $1<p<\8$ and every UMD space $E$?
 \end{problem}

\medskip\n{\bf Acknowledgements.} We thank Chritian Le Merdy for useful comments, Michael Cowling for pointing to us some related references and the referee for helpful suggestions. This work is partially supported by ANR-2011-BS01-008-01 and NSFC grant No.  11271292.

\bigskip


\end{document}